\documentclass[10pt]{amsart}
\usepackage{amsmath,graphicx,amssymb,amsthm}
\usepackage{color}
\def\A{\mathcal{A}}

\def\F{\mathcal F}

\def\C{\mathcal C}
\def\H{\mathcal H}

\def\L{\mathcal L}

\def\P{\mathcal P}

\def\X{\mathcal X}

\def\amslatex{$\mathcal{A}\kern-.1667em\lower.5ex\hbox{$\mathcal{M}$}\kern-.125em\mathcal{S}$-\LaTeX}

\newtheorem{set}{set}[section]
\newtheorem{Corollary}[set]{Corollary}
\newtheorem{Definition}[set]{Definition}

\newtheorem{Remark}[set]{Remark}
\newtheorem{Theorem}[set]{Theorem}
\newcommand{\define}{\mathrel{\hbox{$\equiv$\hskip -.90em \lower .47ex \hbox{$\leftharpoondown$}}}}
\newcommand{\enifed}{\mathrel{\hbox{$\equiv$\hskip -.90em \lower .47ex \hbox{$\rightharpoondown$}}}}

\begin{document}
\title {\bf   Compound Bi-free  Poisson Distributions}
\author{Mingchu Gao}
\address{School of Mathematics and Information Science,
Baoji University of Arts and Sciences,
Baoji, Shaanxi 721013,
China; and
Department of Mathematics,
Louisiana College,
Pineville, LA 71359, USA}
\email{mingchu.gao@lacollege.edu}

\date{}
\begin{abstract}In this paper, we study  compound bi-free Poisson distributions for {\sl two-faced families of random variables}. We prove a Poisson limit theorem for compound bi-free Poisson distributions. Furthermore, a bi-free infinitely divisible distribution for a two-faced family of self-adjoint random variables can be realized as the limit of  a sequence of compound bi-free Poisson distributions of two-faced families of self-adjoint random variables. If a compound bi-free Poisson distribution is determined by a positive number and the distribution of a two faced family of finitely many random variables, which has an almost sure random matrix model, and the left random variables commute with the right random variables in the two-faced family, then we can construct a random bi-matrix model for the compound bi-free Poisson distribution. If a compound bi-free Poisson distribution is determined by a positive number and the distribution of a commutative pair of random variables, we can construct an asymptotic bi-matrix model  with entries of creation and annihilation operators for the compound bi-free Poisson distribution.
\end{abstract}
\maketitle
{\bf AMS Mathematics Subject Classification (2010)} 46L54.

{\bf Key words and phrases} Compound Bi-free Poisson Distributions, Bi-free Infinitely Divisible Distributions, Random Bi-matrix Models, Bi-matrix Models with  Fock Space Entries.
\section*{Introduction}
Studying free probability analogues to classical probability items has been a main topic in free probability since the inception of the theory by Voiculescu in \cite{DV3}.   The most popular and thoroughly studied item is  semicircular distributions, a free probability analogue to classical Gaussian distributions. Free Poisson distributions hold a status of the most popular next to semicircular distributions in free probability. Multidimensional semicircular distributions were first treated in \cite{RS1}. Finite dimensional multi-variable free Poisson distributions were defined and studied in \cite{RS}. More general infinitely dimensional multi-variable compound free Poisson distributions were studied recently in \cite{AG}.

Voiculescu's seminal work in \cite{DV} started a new research field in free probability, bi-free probability. Generalizing the ideas and results in free probability to this new setting has been a main theme in the quick development of bi-free probability. For example, in \cite{DV}, Voiculescu determined bi-free central limit distributions, a bi-free analogue of semicircular distributions. Voiculescu \cite{DV4} constructed a bi-free partial $R$-transform as an analogue of his $R$-transform in \cite{DV5}. Voiculescu \cite{DV2} developed a bi-free partial $S$-transform. On the combinatorial side, Mastnak and Nica  \cite{MN} introduced a collection of partitions for bi-free pairs of faces that was postulated to be analogous to the role non-crossing partitions played in free probability. In \cite{CNS1} the postulate of Mastnak and Nica was confirmed to be true. Subsequently, \cite{CNS2} generalized such notions to the operator-valued setting. In addition, the notion of bi-free infinitely (additive) divisible distributions for commutative pairs of  self-adjoint random variables was developed in \cite{GHM}, which was generalized to the case of (not necessarily commutative) pairs of  random variables in \cite{MG}. Gu, Huang, and Mingo \cite{GHM} discovered compound bi-free Poisson distributions for commutative pairs of random variables in terms of their bi-free cumulants (Example 3.13 (3) in \cite{GHM}).
In this paper, inspired by the work in \cite{RS} and \cite{AG}, we define and study  compound  bi-free Poisson distributions for  {\sl two-faced families of random variables}.

Voiculescu \cite{DV} defined bi-free central limit distributions for two-faced families of random variables in terms of their bi-free cumulants (Definition 7.4 in \cite{DV}). A similar method was adopted in defining (multi-variable) free Poisson distributions in \cite{NS}, \cite{RS} and \cite{AG}. Based on the same philosophy,  we define {\sl compound bi-free Poisson distributions for two-faced families of random variables} in terms of their bi-free cumulants.  A compound bi-free  Poisson distribution  can be realized as the Poisson limit distribution of a sequence of two-faced families of random variables.   Like in free probability,  a bi-free infinitely divisible distribution, as a positive linear functional on the $*$-algebra of polynomials, could be approximated in distribution by  compound bi-free Poisson distributions of self-adjoint random variables.

Discovering of  the connections between free probability and random matrices initially done by Voiculescu in \cite {DV1} is a milestone in the development of free probability, which  led free probability out of the umbrella of operator algebras and into a wider field of research. The development of the connection theory between free probability and random matrices has made free probability to be a powerful tool in random matrices (\cite{AGZ}). Inspired by the corresponding work in free probability, Skoufranis \cite{PS} constructed bi-matrix models for bi-free central limit distributions and a special kind of bi-free Poisson distributions. Precisely, Skoufranis constructed bi-matrix models of random matrices, and of matrices with entries of creation and annihilation operators,  for bi-free central limit distributions (Sections 4 and 5 in \cite{PS}). Skoufranis also constructed a bi-matrix model of random matrices for a bi-free Poisson distribution determined by numbers $\lambda>0, \alpha, \beta \in \mathbb{R}$, i. e., a distribution $\mu$ characterized by $\kappa_\chi(\mu)=\lambda \alpha^{|\chi^{-1}(\{l\})|}\beta^{|\chi^{-1}(\{r\})|}$, for $n\in \mathbb{N}$ and $\chi:\{1, 2, \cdots, n\}\rightarrow \{l, r\}$ (Example 4.15 in \cite{PS}). In this paper, we construct bi-matrix models for a compound bi-free Poisson distribution determined by $\lambda>0$ and the distribution $\mu_a$ of  a bipartite two-faced family $a=((a_{1, l}, \cdots, a_{N, l}), (a_{1, r}, \cdots, a_{M, r}))$, where the tuple $\{a_{1, l}, \cdots, a_{N, l}, a_{1, r}, \cdots, a_{M, r}\}$ has an almost sure random matrix model.

This paper is organized as follows. Besides this Introduction, there are five sections in this paper. In Section 1, we recall some essential combinatorial aspects of bi-free probability, operator-valued bi-free probability, and the general construction of bi-matrix models in \cite{PS}. We give the definition and a Poisson limit theorem for a two-faced family of random variables to have a compound bi-free Poisson distribution in Section 2 (Definition 2.1 and Theorem 2.3). Bi-free infinite divisibility of  distributions of  two-faced families of self-adjoint random variables is defined and characterized by existence of a bi-free additive convolution semigroup associated with the distribution (Definition 3.1 and Theorem 3.4).  Furthermore, such a distribution can be approximated by compound bi-free Poisson distributions of two-faced families of self-adjoint random variables (Theorem 3.5). Historically, Marchenko-Pastur's work in \cite{MP} implies that Wishart matrices with appropriate normalization and scaling form a random matrix model for free Poisson distributions (see also Theorem 4.1.9 in \cite{HP}). F. Hiai and D. Pets \cite{HP} provided a random matrix model for compound free Poisson distribution $P(\lambda, \nu)$ (Proposition 4.4.11 in \cite{HP}), where $\lambda>0$ and $\nu$ is a probability measure on $\mathbb{R}$ with a compact support. In Section 4, we construct a random matrix model for a multi-dimensional compound free Poisson distribution $P(\lambda, \nu_a)$ (Theorem 4.2), where $\lambda>0$, $\nu_a$ is the distribution of $a=(a_1, \cdots, a_N)$, if the tuple $\{a_1, \cdots, a_N\}$ has an almost sure random matrix model. Especially, $P(\lambda, \nu_a)$ has a random matrix model, if the tuple $\{a_1,..., a_N\}$ is classically independent (Corollary 4.3). Applying this result to two-faced families of random variables, we get a random bi-matrix model for a compound bi-free Poisson distribution determined by $\lambda>0$ and $\nu_a$, $a=((a_{1,l},\cdots, a_{N,l}), (a_{1, r}, \cdots, a_{M, r}))$, if the tuple has an almost sure random matrix model, and the left random variables commute with the right random variables in the tuple (Theorem 4.4). P. Skoufranis \cite{PS} constructed bi-matrix models with entries of creation and annihilation operators for bi-free central limit distributions (Theorem 5.1 and Remark 5.2 in \cite{PS}).
Using the operator model in \cite{MN}, we  construct an asymptotic bi-matrix model with entries of creation and annihilation operators for a compound bi-free Poisson distribution determined by $\lambda>0$ and the distribution of a commutative pair $(a_1, a_2)$ (Theorem 5.1).
\section{Preliminaries}
In this section we will summarize some essential combinatorial aspects of bi-free probability, operator-valued bi-free probability,  and  the general construction of bi-matrix models in \cite{PS}.  The reader is referred to \cite{DV}, \cite{CNS1}, \cite{CNS2}, and \cite{PS1} for more details on bi-free probability, and to \cite{NS} and \cite{VDN} for basics on free probability.
\subsection{Combinatorics of Bi-free Probability}
 Let $\chi:\{1, 2, \cdots, n\}\rightarrow  \{l,r\}$. Let's record explicitly where are the occurrences of $l$ and $r$ in $\chi$. $\chi^{-1}(l)=\{i_1<i_2<\cdots <i_p\}$ and $\chi^{-1}(r)=\{i_{p+1}>i_{p+2}>\cdots >i_n\}$. Then we define a permutation $s_\chi$ on $\{1, 2, \cdots, n\}: s_\chi(k)=i_k$, for $k=1, 2,\cdots, n$.
 For a subset $V=\{i_1< i_2< \cdots< i_k\}$ of the set $\{1, 2, \cdots, n\}$, $a_1\cdots, a_n\in \A$, define $$\varphi_V(a_1, \cdots, a_n)=\varphi(a_{i_1}a_{i_2}\cdots a_{i_k}).$$  Let $\P(n)$ be the set of all partitions of $\{1, 2, \cdots, n\}$. For a partition $\pi=\{V_1, V_2, \cdots, V_d\}\in \P(n)$, we define $$\varphi_\pi(a_1, \cdots, a_n):=\prod_{V\in \pi}\varphi_V(a_1, \cdots, a_n).$$  Define $BNC(n,\chi)=\{s_\chi\circ \pi:\pi\in NC(n)\}$, where $NC(n)$ is the set of all non-crossing partitions of $\{1, 2, \cdots, n\}$ (Lecture 9 in \cite{NS}).  A $\sigma\in BNC(n,\chi)$ is called a {\sl bi-non-crossing} partition of $\{1, 2, \cdots, n\}$.  Let $(\A,\varphi)$ be a non-commutative probability space. The bi-free cumulants $(\kappa_\chi:\A^n\rightarrow \mathbb{C})_{n\ge 1, \chi:\{1, 2, \cdots, n\}\rightarrow \{l,r\}}$ of $(\A,\varphi)$ are defined by
 $$\kappa_\chi (a_1, \cdots, a_n)=\sum_{\pi\in BNC(n,\chi)}\varphi_{\pi}(a_1, \cdots, a_n)\mu_n(s^{-1}_\chi\circ\pi, 1_n),\eqno (1.1)$$ for $n\ge 1, \chi:\{1,2,\cdots, n\}\rightarrow \{l,r\}, a_1, \cdots, a_n\in A$, where $\mu_n$ is the Mobius function on $NC(n)$ (Lecture 10 in \cite{NS}). For a subset $V=\{i_1, i_2, \cdots, i_k\}\subseteq \{1, 2, \cdots, n\}$, let $\chi_V$ be the restriction of $\chi$ on $V$. We define $\kappa_{\chi, V}(a_1, a_2, \cdots, a_n)=\kappa_{\chi_V}(a_{i_1}, a_{i_2}, \cdots, a_{i_k})$.
 For a partition $\pi=\{V_1, V_2, \cdots, V_k\} \in BNC(n,\chi)$, we define $\kappa_{\chi, \pi}(a_1, \cdots, a_n)=\prod_{V\in \pi}\kappa_{\chi, V}(a_1, a_2, \cdots, a_n)$.  Then the bi-free cumulant appeared in $(1.1)$ is $\kappa_{\chi, 1_n}(a_1, \cdots, a_n)$. The bi-free cumulants are determined by the equation $$\varphi (a_1 a_2 \cdots a_n)=\sum_{\pi\in BNC(n, \chi)}\kappa_{\chi,\pi}(a_1, a_2, \cdots, a_n), \forall a_1, \cdots, a_n\in \A,   \eqno (1.2)$$ for a $\chi:\{1, 2, \cdots, n\}\rightarrow \{l,r\}$.

 Charlesworth, Nelson, and Skoufranis \cite{CNS1} proved that
  $$z'=((z'_i)_{i\in I}, (z'_j)_{j\in J}), z''=((z''_i)_{i\in I}, (z''_j)_{j\in J})$$ in a non-commutative probability space $(\A,\varphi)$ are bi-free if and only if
 $$\kappa_\chi(z_{\alpha(1)}^{\epsilon_1},z_{\alpha(2)}^{\epsilon_2}, \cdots, z_{\alpha(n)}^{\epsilon_n})=0, \eqno (1.3)$$ whenever $\alpha:\{1,2,\cdots, n\}\rightarrow I\bigsqcup J$, $\chi:\{1, 2, \cdots, n\}\rightarrow \{l,r\}$ such that $\alpha^{-1}(I)=\chi ^{-1}(\{l\})$,   $\epsilon:\{1,2, \cdots, n\}\rightarrow \{',''\}$ is not constant, and $n\ge 2$ (Theorem 4.3.1 in \cite{CNS2}).

 Let $\mu$ and $\nu$ be distributions of the pairs $(a_l, a_r)$ and $(b_l, b_r)$, respectively. We call the distribution $\mu\boxplus\boxplus\nu$ of $(a_l+b_l, a_r+b_r)$ the bi-free additive convolution of $\mu$ and $\nu$, if $(a_l, a_r)$ and $(b_l, b_r)$ are bi-free.
\subsection{Structures of Operator-valued Bi-freeness}
Let $B$ be a unital algebra.
 A $B$-$B$-{\sl non-commutative probability space} is a triple $(\A,E,\varepsilon)$ where $\A$ is a unital algebra, $\varepsilon: B\otimes B^{op}\rightarrow \A$ is a unital homomorphism such that $\varepsilon|_{B\otimes 1_B}$ and $\varepsilon|_{1_B\otimes B}$ are injective, and $E:\A\rightarrow B$ is a linear map such that $E(\varepsilon(b_1\otimes b_2)a)=b_1E(a)b_2$ and $E(a\varepsilon(b\otimes 1_B))=E(a\varepsilon(1_B\otimes b))$. Let $L_b=\varepsilon(b\otimes 1_B)$ and $R_b=\varepsilon(1_B\otimes b)$. The unital subalgebras $$\A_l=\{a\in \A: aR_b=R_ba, \forall b\in B\}, \A_r=\{a\in \A: L_ba=aL_b,\forall b\in B\}$$ are called the {\sl left } and {\sl right} algebras of $\A$, respectively.
The following we give a canonical example of $B$-$B$-non-commutative probability spaces.

A {\sl $B$-$B$-bi-module with a specified $B$-vector state} is a triple $(\mathcal{X}, \mathcal{X}^0, p)$ where $\mathcal{X}=B\oplus \mathcal{X}^0$, a direct sum of $B$-$B$ bi-modules and $p:\mathcal{X}\rightarrow B$, $p(b\oplus\eta)=b$. Let $\L(\X)$ denote the set of linear operators on $\mathcal{X}$. For $b\in B$, define $L_b, R_b\in \L(\X)$ by $$L_b(x)=bx, R_b(x)=xb, \forall x\in \X.$$ Similarly, we can define left and right algebras as follows $$\L_l(\X):=\{A\in \L(\X): AR_b=R_bA, \forall b\in B\}, \L_r(\X):=\{A\in \L(\X):AL_b=L_bA, \forall b\in B\}. $$ Given a $B$-$B$-bi-module with a specified $B$-vector state $\{\X, \X^0, p\}$, the expectation $E_{\L(\X)}$ of $\L(\X)$ onto $B$ is defined by $$E_{\L(\X)}(A)=p(A1_B), \forall A\in \L(\X).$$ Define $\varepsilon: B\otimes B^{op}\rightarrow \A$, $\varepsilon(b_1\otimes b_2)=L_{b_1}R_{b_2}$. Then $(\L(\X), E_{\L(\X)}, \varepsilon)$ is a (concrete) $B$-$B$-non-commutative probability space. Moreover, Theorem 3.2.4 in \cite{CNS2} demonstrated that every abstract $B$-$B$-non-commutative probability space can be represented inside a concrete $B$-$B$-non-commutative probability space.
\subsection{Bi-Matrix Models}
Let $\X$ be a vector space over $\mathbb{C}$ with $\X=\mathbb{C}\xi\oplus \X^0$ and $p:\X\rightarrow \mathbb{C}, p(\lambda\xi+\eta)=\lambda$. We call $(\X, \X^0, \xi, p)$ a pointed vector space. For $N\in \mathbb{N}$, consider $M_N(\mathbb{C})$-$M_N(\mathbb{C})$ bi-modular actions on $\X_N:=M_N(\X)$, $$[a_{i,j}]\cdot[\eta_{i,j}]=[\sum_{k=1}^Na_{i,k}\eta_{k,j}], [\eta_{i,j}]\cdot[a_{i,j}]=[\sum_{k=1}^Na_{k,j}\eta_{i,k}],$$ for all $[a_{i,j}]\in M_N(\mathbb{C})$ and $[\eta_{i,j}]\in \X_N$. Then $\X_N$ becomes an $M_N(\mathbb{C})$-$M_N(\mathbb{C})$-bi-module with a specified $M_N(\mathbb{C})$-vector state via $$\X_N=M_N(\mathbb{C}\xi)\oplus M_N(\X^0),$$
and a linear map $p_{\X_N}:\X_N\rightarrow M_N(\mathbb{C})$ defined by $p_{\X_N}([\eta_{i,j}])=[p(\eta_{i,j})]$, which  is called the $M_N(\mathbb{C})$-$M_N(\mathbb{C})$ -bi-module associated with $(\X,p)$ and $(\L(\X_N), E_{\L(\X_N)}, \varepsilon)$ is called the $M_N(\mathbb{C})$-$M_N(\mathbb{C})$-non-commutative probability space associated with $(\X,p)$. The expectation $E_{\L(\X_N)}:\L(\X_N)\rightarrow M_N(\mathbb{C})$ has the form $E_{\L(\X_N)}(A)=p_{\X_N}(A1_{N,\xi})$, where $A\in \L(\X_N)$ and $1_{M,\xi}$ is the diagonal matrix $diag(\xi, \xi, \cdots, \xi)$.

To consider bi-matrix models, we define two homomorphisms $L: M_N(\L(\X))\rightarrow \L(\X_N)$, and $R: M_N(\L(\X)^{op})^{op}\rightarrow \L(\X_N)$, $$L([T_{i,j}])[\eta_{i,j}]=[\sum_{k=1}^NT_{i,k}(\eta_{k,j})], R([T_{i,j}])[\eta_{i,j}]=[\sum_{k=1}^NT_{k,j}\eta_{i,k}],$$ for $[\eta_{i,j}]\in \X_N$ and $[T_{i,j}]\in M_N(\L(\X))$. $L([T_{i,j}])$ and $R([T_{i,j}])$ are called left and right matrices of $\L(\X)$, respectively.

Let $\A=(L^{\infty-}(\Omega, P), E)$, where $L^{\infty-}(\Omega, P)=\cap_{p\ge 1}L^p(\Omega, P)$, $(\Omega, P)$ is a probability space, and $E: f\mapsto \int_\Omega f(d)dP(t)$ is the expectation. Skoufranis \cite{PS} introduced {\sl random pairs of matrices} as follows. For $N\in \mathbb{N}$, an $N\times N$ random pair of matrices on $L^{\infty-}(\Omega, P)$ is a pair $(X_l, X_r)$, where $X_l$ is a left matrix and $X_r$ is a right matrix with entries from $\A\subset \L(L^2(\Omega, P))$ (Definition 4.5 in \cite{PS}). We generalize the concept to two-faced families.
\begin{Definition}
For $N\in \mathbb{N}$, an $N\times N$ random two-faced family of matrices on $L^{\infty-}(\Omega, P)$ is a two-faced family $((X_{i})_{i\in I}, (X_{j})_{j\in J})$ where $X_i, i\in I,$ are left matrices and $X_j, j\in J,$ are right matrices with entries from $L^{\infty-}(\Omega, P)$.
\end{Definition}

{\bf Acknowledgements} The author would like to thank the referee(s) for carefully reading the original version of this paper, pointing out some mistakes and typos in it,  and providing some suggestions to revise it.
\section{ The Definition and A Poisson Limit Theorem}

In this section, we give the definition for a two-faced family to have a  bi-free compound Poisson distribution.    Furthermore, a bi-free compound Poisson distribution can be realized via a bi-free Poisson limit theorem.

\begin{Definition} Let $I$ and $J$ be two disjoint index sets, and $((z_i)_{i\in I}, (z_j)_{j\in J})$ be a two-faced family of random variables in a non-commutative probability space $(\A,\varphi)$. We say that $((z_i)_{i\in I}, (z_j)_{j\in J})$ has a compound bi-free  Poisson distribution, if there exist a real number $\lambda>0$ and a two-faced family $((a_i)_{i\in I}, (a_j)_{j\in J})$ of random variables in $(\A,\varphi)$ such that, for every $n\in \mathbb{N}$, and a map $$\chi:\{1, 2, \cdots, n\}\rightarrow I\bigsqcup J,$$ we have  $$\kappa_\chi(z_{\chi(1)}, \cdots, z_{\chi(n)})=\lambda \varphi (a_{\chi(1)}a_{\chi(2)}\cdots a_{\chi(n)}).$$ We call the distribution of $((z_i)_{i\in I}, (z_j)_{j\in J})$ a compound bi-free  Poisson distribution determined by $\lambda$ and $\mu_a$, the distribution of $a=((a_i)_{i\in I}, (a_{j})_{j\in J})$.
\end{Definition}
\begin{Remark} The above definition covers the following well-known cases.

\begin{enumerate}
 \item Let $I=\{l\}, J=\{r\}$, $[z_l, z_r]=[a_l, a_r]=0$. Then $$\kappa_{\chi}(z_l, z_r)=\lambda \varphi(a_l^{|\chi^{-1}(\{l\})|}a_r^{|\chi^{-1}(\{r\})|}), $$ which defines a compound  bi-free Poisson distribution for a commutative pair of random variables (Example 3.13 (3) of \cite{GHM}).
 \item If $\chi:\{1,2,\cdots, n\}\rightarrow I$, we get $$\kappa(z_{\chi(1)}, \cdots, z_{\chi(n)})=\lambda \nu (X_{\chi(1)}X_{\chi(2)}\cdots X_{\chi(n)}),$$ where $\nu:\mathbb{C}(X_i:i\in I)\rightarrow \mathbb{C}$ defined by $\nu(X_{\chi(1)}X_{\chi(2)}\cdots X_{\chi(n)})=\varphi (a_{\chi(1)}a_{\chi(2)}\cdots a_{\chi(n)})$ is the distribution of $\{a_i:i\in I\}$ in $(\A, \varphi)$. In this case, we obtain the scalar-valued compound (free) Poisson distributions defined in 4.4.1 of \cite{RS} with parameter $\lambda>0$. It follows that Definition 2.1 provides a bi-free analogue of multi-variable compound free Poisson distributions.
\end{enumerate}
\end{Remark}
Corollary 2.4 in \cite{MG} gives the following Poisson limit theorem for compound bi-free  Poisson distributions.
\begin{Theorem} Let $((a_i)_{i\in I}, (a_j)_{j\in J})$ be a two-faced family of random variables in $(\A, \varphi)$, and $\lambda >0$. For each $N>\lambda$, let $(\C_N, \phi_N)$ be a $C^*$-probability space, and $p_N\in \mathcal{C}_N$ be a projection with $\phi_N(p_N)=\frac{\lambda}{N}$. Let $a_N= ((a_i\otimes p_{N})_{i\in I}, (a_j\otimes p_{N})_{j\in J})$ be a two-faced family of random variables in $\A_N:=\A\otimes \C_N$ with $\varphi_N:=\varphi \otimes \phi_N$. Let $\{(a_{N, i, m})_{i\in I}, (a_{N,j,m})_{j\in J}):m=1, 2, \cdots, N\}$ be a bi-free sequence of $N$ identically distributed two-faced families of random variables in $(\A_N, \varphi_N)$ such that each of the two-faced families has the same distribution as that of $a_N$. Let,  finally, $S_{N,k}=\sum_{m=1}^Na_{N,k,m}$, for $k\in I\bigsqcup J$,  $S_N=(S_{N,i})_{i\in I},(S_{N,j})_{j\in J})$. Then $S_N$ converges in distribution to the compound bi-free  Poisson distribution determined by $((a_i)_{i\in I}, (a_j)_{j\in J})$ and $\lambda$, $N\rightarrow \infty$.
\end{Theorem}
\begin{proof} Let $n\in \mathbb{N}$, and $\chi:\{1, 2, \cdots, n\}\rightarrow I\bigsqcup J$. By Corollary 2.4 in \cite{MG}, the limit distribution of $S_N$, as $N\rightarrow \infty$, is characterized by
\begin{align*}
\kappa_\chi(b_{\chi(1)}, \cdots, b_{\chi(n)})&=\lim_{N\rightarrow \infty}N\varphi_N((a_{\chi(1)}\otimes p_{N})\cdots (a_{\chi(n)}\otimes p_{N}))\\
=&\lim_{N\rightarrow \infty}N\cdot \frac{\lambda}{N}\varphi(a_{\chi(1)}a_{\chi(2)}\cdots a_{\chi(n)})\\
=&\lambda \varphi(a_{\chi(1)}a_{\chi(2)}\cdots a_{\chi(n)}),
\end{align*}where the two-faced family $((b_i)_{i\in I}, (b_j)_{j\in J})$ has the limit distribution.
\end{proof}
In the $C^*$-probability space case, when taking the positivity of the state $\varphi$ on $\A$ into account, we can get a conclusion that the bi-free compound Poisson distribution $P_{\lambda, \mu_a}$ is  a positive linear functional on the polynomial algebra.
\begin{Corollary} Let $(\A, \varphi)$ be a $C^*$-probability space, and $I$ and $J$ be disjoint index sets. Let $\lambda>0$ be a positive number and $a:=((a_i)_{i\in I}, (a_j)_{j\in J})$ be a two-faced family of self-adjoint random variables in $\A$. Then the bi-free compound Poisson distribution $P_{\lambda, \mu_a}:\mathbb{C}\langle X_k: k\in I\bigsqcup J\rangle \rightarrow \mathbb{C}$ is positive, where $\mathbb{C}\langle X_k: k\in I\bigsqcup J\rangle$ is a unital $*$-algebra with $X_k=X^*_k$, for $k\in I\bigsqcup J$.
\end{Corollary}
\begin{proof}
For a two-faced family $c=((c_i)_{i\in I}, (c_j)_{j\in J})$ of self-adjoint random variables in $\A$, $\mu_c:\mathbb{C}\langle X_k: k\in I\bigsqcup J\rangle \rightarrow \mathbb{C}$ is positive. In fact, for a polynomial $P=\sum \alpha_kX_k\in \mathbb{C}\langle X_k: k\in I\bigsqcup J\rangle$, we have $$\mu(c)(P^*P)=\mu_c(\sum_{k, k'}\overline{\alpha_k}\alpha_{k'}X_{k}X_{k'})=\sum_{k,k'}\overline{\alpha_k}\alpha_{k'}\varphi(c_{k}c_{k'})=\varphi((\sum_{k, k'}\alpha_kc_k)^*(\sum_{k,k'}\alpha_kc_k))\ge 0.$$ Moreover, if $\mu_1$ and $\mu_2$ are positive linear functionals on $\mathbb{C}\langle X_k: k\in I\bigsqcup J\rangle$, then $\mu_1\boxplus\boxplus\mu_2$ is positive. Let $\mu_N$ be the distribution of $((a_i\otimes p_N)_{i\in I}, (a_j\otimes p_N)_{j\in J})$ of self-adjoint operators in $\A_N=\A\otimes_m \C_N$, the spacial tensor product of $C^*$-algebras $\A$ and $\C_N$, in the proof of Theorem 2.3. Then $\mu_{S_N}=\mu_N^{\boxplus\boxplus N}$ is positive.  By Theorem 2.3, $P_{\lambda, \mu_a}$ is the weak limit of $\mu_{S_N}$. It implies that the bi-free Poisson distribution $P_{\lambda, \mu_a}$ is positive.
\end{proof}
\section{Bi-free Infinitely Divisible Distributions}
Bi-free infinite divisibility of the distribution of a pair $(a,b)$ of random variables was defined and studied in \cite{GHM} and \cite{MG}. We now generalize the concept to a more general case of two-faced families. Like in free probability (see Section 4.5 in \cite{RS}), a bi-free infinitely divisible distribution can be approached by compound bi-free Poisson distributions.

\begin{Definition} When $X_k=X^*_k$ for $k\in I\bigsqcup J$, $\mathbb{C}\langle X_k: k\in I \bigsqcup J\rangle$ becomes a unital $*$-algebra. In this case, we use $\Sigma^+(I, J)$ to denote the set of all positive unital linear functional on $\mathbb{C}\langle X_k: k\in I \bigsqcup J\rangle$.
 \begin{enumerate}
\item A distribution $\mu\in \Sigma^+(I, J)$ is bi-free infinitely divisible if for each $N\in \mathbb{N}$, there is a distribution $\mu_{1/N}\in \Sigma^+(I,J)$ such that $$\mu=\underbrace{\mu_{1/N}\boxplus\boxplus\mu_{1/N}\boxplus\boxplus\cdots\boxplus\boxplus\mu_{1/N}}_{N \text{ times}}.$$ In the language of random variables, we have the following equivalent definition.
\item A two-faced family $((z_i)_{i\in I}, (z_j)_{j\in J})$ of self-adjoint random variables in a $*$-probability space $(\A, \varphi)$ has a bi-free infinitely divisible distribution if for each $N\in \mathbb{N}$, there exits a bi-free sequence of $N$ identically distributed two-faced families $\{((z_{N, i, n})_{i\in I}, (z_{N,j, n})_{j\in J}):n=1, 2, \cdots, N \}$ of self-adjoint random variables  such that $((S_{N,i})_{i\in I}, (S_{N,j})_{j\in J})$ has the same distribution as that of $((z_i)_{i\in I}, (z_j)_{j\in J})$, where $S_{N,k}=\sum_{n=1}^Nz_{N, k, n}$ for $k\in I\bigsqcup J$.
\end{enumerate}
\end{Definition}
\begin{Remark}
There is a one to one correspondence between the set of distributions of two-faced families $((z_i)_{i\in I}, (z_j)_{j\in J})$ in  some non-commutative probability space $(\A, \varphi)$ and the set $\Sigma(I, J)$ of all unital linear functional on the unital polynomial algebra $\mathbb{C}(X_k:k\in I\bigsqcup J)$ in non-commutative variables $\{X_k:k\in I\bigsqcup J\}$.

In fact, for $z:=((z_i)_{i\in I}, (z_j)_{j\in J})$ in $(\A, \varphi)$, define $\nu_z:\mathbb{C}(X_k:k\in I\bigsqcup J)\rightarrow \mathbb{C}$ by $$\nu_z(P(X_k:k\in I\bigsqcup J))=\varphi(P(z_k:k\in I \bigsqcup J)).$$ It is obvious that $\nu_z\in \Sigma(I\bigsqcup J)$. Conversely, Let $\nu\in \Sigma(I\bigsqcup J)$. Then $(\mathbb{C}(X_k:k\in I\bigsqcup J), \nu)$ is a non-commutative probability space, and $\nu$ is the distribution of $((X_i)_{i\in I}, (X_j)_{j\in J})$. Similarly, we can identify $\Sigma^+(I, J)$ with distributions of two-faced families of self-adjoint random variables.
Note that $\Sigma^+(I, J)$ is a convex set of linear functionals. Therefore, we can define $\lambda \nu_1 +(1-\lambda)\nu_2\in \sum(I, J)$, for $0\le \lambda \le 1$ and $\nu_1, \nu_2\in \sum(I, J)$.
\end{Remark}

Lemma 5.2 in \cite{GHM} states that if $a_1^*=a_1, a_2^*=a_2, [a_1, a_2]=0$, for $a_1, a_2$ in a $C^*$-probability space $(\A, \varphi)$, a projection $p\in \A$ free from $\{a_1, a_2\}$, then there exists a compactly supported probability measure $\mu$ on $\mathbb{R}^2$ such that  $$\kappa_{m,n}^\mu=\kappa_{m,n}^{p\A p}(\underbrace{pa_1p, \cdots, pa_1p}_{m \ \mathrm{ times}}, \underbrace{pa_2p, \cdots, pa_2p}_{n\  \mathrm{ times}}).$$ The following result gives a formula to compute the bi-free cumulants of a free projection compressed family, which is a bi-free analogue of Theorem 14.10 in \cite{NS}.
\begin{Theorem} Let $z=((z_i)_{i\in I}, (z_j)_{j\in J})$ be a two-faced family of random variables in a non-commutative -probability space $(\A, \varphi)$, and $p$ is a projection ($p=p^2$) in $\A$, which is free from $\{z_k:k\in I \bigsqcup J\}$,  and $\varphi(p)=\lambda\ne 0$. Then for a number $n\in \mathbb{N}$ and a function $\chi:\{1, 2, \cdots, n\}\rightarrow I\bigsqcup J$, we have $$\kappa_\chi^{p\A p}(pz_{\chi(1)}p, \cdots, pz_{\chi(n)}p)=\frac{1}{\lambda}\kappa_\chi(\lambda z_{\chi(1)}, \cdots, \lambda z_{\chi(n)}),$$ where $\kappa_\chi$ and $\kappa_\chi^{p\A p}$ are the bi-free cumulants of $(\A, \varphi)$ and $(p\A p, \varphi_p)$, respectively.
\end{Theorem}
\begin{proof}Let $\chi:\{1, 2, \cdots, n\}\rightarrow I\bigsqcup J$, $\chi^{-1}(I)=\{i_1<i_2<\cdots i_k\}$ and $$\chi^{-1}(J)=\{i_{k+1}>i_{k+2}>\cdots >i_{n}\}.$$ As in Section 1, we define a permutation $s_\chi\in S_n$, $s_\chi(j)=i_j$, for $j=1, 2, \cdots, n$. The permutation $s_\chi $ defines a lattice isomorphism from $NC(n)$ onto $BNC(\chi)$ by $\pi \mapsto s_\chi\cdot\pi$, for $\pi\in NC(n)$, where $$s_\chi\cdot \pi=\{s_\chi\cdot V=\{s_\chi(t_1), s_\chi(t_2), \cdots, s_\chi(t_k)\}: V=\{t_1, t_2, \cdots, t_k\}\in \pi\}.$$ Thus, $s_\chi\cdot \pi$ is the corresponding partition of $\pi$ in the new partial ordered set $$\{s_\chi(1)\prec s_\chi(2)\prec\cdots \prec s_\chi(n)\}.$$ It implies that $\varphi_\pi(z_{\chi(s_\chi(1))}, z_{\chi(s_\chi(2))}, \cdots, z_{\chi(s_\chi(n))})=\varphi_{s_\chi\cdot \pi}(z_{\chi(1)}, z_{\chi(2)}, \cdots, z_{\chi(n)})$. Therefore, by $(1.1)$, we have
\begin{align*}
\kappa_n(z_{\chi(s_\chi(1))}, \cdots, z_{\chi(s_\chi(n))})=&\sum_{\pi\in NC(n)}\varphi_\pi(z_{\chi(s_\chi(1))}, \cdots, z_{\chi(s_\chi(n))})\mu_{NC}(\pi, 1_n)\\
=&\sum_{\pi\in NC(n)}\varphi_{s_\chi\cdot\pi}(z_{\chi(1)}, \cdots, z_{\chi(n)})\mu_{BNC(\chi)}(s_\chi\cdot\pi,1_n)\\
=&\sum_{\sigma\in BNC(\chi)}\varphi_\sigma(z_{\chi(1)}, \cdots, z_{\chi(n)})\mu_{BNC(\chi)}(\sigma,1_n))\\
=&\kappa_\chi(z_{\chi(1)}, \cdots, z_{\chi(n)}).
\end{align*}
By Theorem 14.10 in \cite{NS}, we get
\begin{align*}\kappa_\chi^{p\A p}(pz_{\chi(1)}p, \cdots, pz_{\chi(n)}p)=&\kappa_n^{p\A p}(pz_{\chi(s_\chi(1))}p, \cdots, pz_{\chi(s_\chi(n))}p)\\
=&\frac{1}{\lambda}\kappa_n(\lambda z_{\chi(s_\chi(1))}, \cdots, \lambda z_{\chi(s_\chi(n))})\\
=&\frac{1}{\lambda}\kappa_{\chi}(\lambda z_{\chi(1)}, \cdots, \lambda z_{\chi(n)}).
\end{align*}
\end{proof}

The following theorem characterizes the bi-free infinite divisibility of a distribution in $\Sigma^+(I,J)$ in terms of the existence of a bi-free additive convolution semigroup of distributions in $\Sigma^+(I, J)$ associated with the distribution.
 \begin{Theorem} Let $\nu\in \Sigma^+(I, J)$. Then the distribution $\nu$ is bi-free infinitely divisible if and only if there is a semigroup $\{\nu_t: t\ge 0\}$ of distributions in $\Sigma^+(I, J)$ such that $\nu_{s+t}=\nu_s\boxplus\boxplus\nu_t$ for $s, t\ge 1$, $\nu_1=\nu$,  $\nu_0=\delta_0$, and $\lim_{t\rightarrow 0}\nu_t=\delta_0$ weakly.
 \end{Theorem}
\begin{proof}

Let $\nu$ be a bi-free infinitely divisible distribution in $\Sigma^+(I,J)$. Then, for every $2\le n\in \mathbb{N}$, there is a two-faced family $z_{1/n}=((z_{i, 1/n})_{i\in I}, (z_{j,1/n})_{j\in J})$ of self-adjoint random variables in a $*$-probability space $(\A_{1/n}, \varphi_{1/n})$ such that $\nu_{1/n}^{\boxplus\boxplus n}=\nu$, where $\nu_{1/n}\in \Sigma^+(I, J)$ is the distribution of $z_{1/n}$, Therefore, for $\chi: \{1, 2, \cdots, n\}\rightarrow I\bigsqcup J$, we have $\kappa_{\chi, z_{1/n}}=\frac{1}{n}\kappa_{\chi,z}$. By performing bi-free additive convolution, we can get  a distribution $\nu_{r}\in \Sigma^+(I, J)$ such that $\kappa_{\chi, \nu_r}=r\kappa_{\chi, z}$, for a positive rational number $r$. For a real number $t>0$, there is a sequence $\{r_n:n=1, 2, \cdots\}$ of positive rational numbers such that $\lim_{n\rightarrow \infty}r_n=t$, we define distribution $\nu_t$ as follows. Whenever $n\in \mathbb{N}, \chi:\{1, 2, \cdots, n\}\rightarrow I\bigsqcup J$, the moment
\begin{align*}
m_{\chi, \nu_t}=\nu_t(X_{\chi(1)}X_{\chi(2)}\cdots X_{\chi(n)}):=&\lim_{n\rightarrow \infty}\varphi_{r_n}(z_{\chi(1)}\cdots z_{\chi(n)})=\lim_{n\rightarrow\infty}\sum_{\pi\in BNC(\chi)}\kappa_{\chi, \nu_{r_n}}\\
=&\lim_{n\rightarrow \infty}\sum_{\pi\in BNC(\chi)}r_n\kappa_{\chi,\nu}=\sum_{\pi\in BNC(\chi)}t\kappa_{\chi, \nu}.
\end{align*} Define $\nu_t(1)=1$. Then $\nu_t\in \Sigma^+(I, J)$ and $\kappa_{\chi, \nu_t}=t\kappa_{\chi, \nu}$. Define $\nu_0=\delta_0$. Then $\{\nu_t:t\ge 0\}$ is semigroup of distributions in $\Sigma^+(I, J)$ with respect to the bi-free additive convolution. Moreover, $\nu_t\rightarrow \delta_0$ weakly, as $t\rightarrow 0+$, since $\lim_{t\rightarrow 0}m_{\chi, \nu_t}=0$, for $n\in \mathbb{N}$ and $\chi:\{1, 2, \cdots, n\}\rightarrow I\bigsqcup J$.

Conversely, if $\{\nu_t: t\ge 0\}$ is such a semigroup,  it is obvious that $\nu_1$ is bi-free infinitely divisible.
\end{proof}
The following theorem gives a bi-free Poisson approach to bi-free infinitely divisible distributions, which is a bi-free analogue of Theorem 4.5.5 in \cite{RS}.
\begin{Theorem} A distribution $\nu\in \Sigma^+(I, J)$ is bi-free infinitely divisible if and only if there is a sequence $\{P_{\lambda_n,\nu_n}:n=1,2,\cdots\}$ of compound bi-free  Poisson distributions determined by  $\lambda_n>0$ and $\nu_n\in \Sigma^+(I, J)$ such that $P_{\lambda_n,\nu_n}\rightarrow \nu$ weakly, as $n\rightarrow \infty$.
\end{Theorem}
\begin{proof}If $P_{\lambda, \nu}$ is a bi-free compound Poisson distribution with  $\nu\in \Sigma^+(I, J)$, then for $n\in \mathbb{N}$ and $\chi:\{1,2,\cdots, n\}\rightarrow I\bigsqcup J$, we have $$\kappa_{\chi,P_{\lambda, \nu}}=\lambda m_{\chi,\nu}=n(\frac{\lambda}{n}m_{\chi, \nu})=n\kappa_{\chi, P_{\lambda/n, \nu}}.$$ It follows that $P_{\lambda, \nu}=(P_{\lambda/n, \nu})^{\boxplus\boxplus n}$ and $P_{\lambda/n, \nu}\in \Sigma^+(I, J)$. Therefore, $P_{\lambda, \nu}$ is bi-free infinitely divisible. Furthermore, If $\nu_k\rightarrow \nu$ weakly in $\Sigma^+(I, J)$, and $\nu_k$ is bi-free infinitely divisible for each $k\ge 1$, then for $n\in \mathbb{N}$, and $\chi:\{1,2,\cdots, n\}\rightarrow I\bigsqcup J$, we have $$\kappa_{\chi,\nu}=\lim_{k\rightarrow \infty}\kappa_{\chi,\nu_k}=n\lim_{k\rightarrow \infty}\frac{1}{n}\kappa_{\chi, \nu_k}=n\lim_{k\rightarrow \infty}\kappa_{\chi, \nu_{k, 1/n}},\eqno (3.1) $$ where $\nu_{k, 1/n}\in \Sigma^+(I, J)$ such that $(\nu_{k, 1/n})^{\boxplus\boxplus n}=\nu_k$, for $k=1, 2, \cdots$. Define $\nu_{1/n}$ as the weak limit of $\nu_{k,1/n}$, as $k\rightarrow \infty$. This weak limit exists and is in $\Sigma^+(I, J)$, because of $(3.1)$. The equation (3.1) also shows that $\kappa_{\chi, \nu}=n\kappa_{\chi, \nu_{1/n}}$. Therefore, $\nu=\nu_{1/n}^{\boxplus\boxplus n}$, that is, $\nu$ is bi-free infinitely divisible.

Conversely, if $\nu\in \Sigma^+(I, J)$ is bi-free infinitely divisible, then $\nu$ can be extended to a semigroup $\{\nu_t\in \Sigma^+(I, J):t\ge 0\}$ with $\nu_1=\nu$, by Theorem 3.4. Thus, we can define a sequence $\{P_{k, \nu_{1/k}}: k=1, 2, \cdots\}$ of bi-free compound Poisson distributions. For $m\in \mathbb{N}$ and $\chi:\{1,2,\cdots, m\}\rightarrow I\bigsqcup J$, we then have
\begin{align*}
\lim_{n\rightarrow\infty}\kappa_{\chi, P_{n, \nu_{1/n}}}=&\lim_{n\rightarrow \infty}nm_{\chi, \nu_{1/n}}=\lim_{n\rightarrow \infty}n\sum_{\sigma\in BNC(\chi)}\kappa_{\chi, \nu_{1/n}, \sigma}\\
=&\lim_{n\rightarrow \infty}n\sum_{\sigma\in BNC(\chi)}\prod_{V\in\sigma}\frac{1}{n}\kappa_{\chi|_{V},\nu}\\
=&\lim_{n\rightarrow \infty}n\sum_{\sigma\in BNC(\chi)}\frac{1}{n^{|\sigma|}}\prod_{V\in \pi}\kappa_{\chi|_{V},\nu}=\kappa_{\chi, \nu}.
\end{align*}
\end{proof}

\section{Random Bi-Matrix Models}

The goal of this section is to construct random bi-matrix families (see Definition 1.1 for the definition) to approximate in distribution to a bi-free compound Poisson distribution $P(\lambda, \mu_a )$, when $a=((a_i)_{i\in I}, (a_{j})_{j\in J})$ has an almost sure random matrix model. Our result will generalize Example 4.15 in \cite{PS} to a much more general case. We first recall a concept from \cite{HP}.

\begin{Definition}[Page 125 in \cite{HP}, Page 19 in \cite{MS}] An $n\times n$ complex self-adjoint random matrix is called a GUE random matrix (GUE stands for Gaussian unitary ensemble) if
\begin{enumerate}
\item $\{\Re X_{ij}(n): 1\le i\le j\le n\}\cup\{\Im X_{ij}(n): 1\le i<j\le n\}$ is an independent family of Gaussian random variables, and
\item $E(X_{ij}(n))=0$, for all $i, j$, $E((X_{ii}(n))^2)=\frac{1}{n}$, for all $i$, and $$E((\Re X_{ij}(n))^2)=E((\Im X_{ij}(n))^2)=\frac{1}{2n},$$ for $1\le i<j\le n$.
\end{enumerate}
\end{Definition}

A tuple $\{X(1,n), ..., X(N,n)\}$ of $n\times n$ random matrices on a probability space $(\Omega, P)$ has {\sl an almost sure limit in distributions}, if there is a tuple $(a_1,..., a_N)$ in a non-commutative probability space $(\A, \varphi)$ such that for almost all $\omega\in \Omega$, $\{X(1,n,\omega), ..., X(N,n,\omega)\}\subset (M_n(\mathbb{C}), tr_n)$ converges in distribution to  $(a_1,..., a_N)$, as $n\rightarrow \infty$. If, furthermore, $\{a_1, ..., a_N\}$ is a free family of random variables, we say that {\sl $X(1,n), ..., X(N,n)$ are almost surely asymptotically  free} (Section 4.1 in \cite{MS}). In this case, we call $\{X(1,n),..., X(N,n))\}$  an almost sure random matrix model of $(a_1,..., a_N)$.

The following theorem gives a random matrix model for a compound free Poisson distribution determined by $\lambda>0$ and a tuple of random variables $a=(a_1,..., a_N)$, generalizing the work on random matrix models of compound free Poisson distributions of single random variables in Section 4.4 of \cite{HP} to the multi-dimensional random variable case.
\begin{Theorem}
Let  $a_1, a_2, \cdots, a_N$ be self-adjoint elements in a $C^*$-probability space $(\A,\varphi)$, and $\lambda>0$ be a positive number. If $\{a_1, ..., a_N\}$ has an almost sure random matrix model, then   there are a subsequence $\alpha(n)$ of natural numbers and a  sequence $\{Y(n, i):i=1, 2, \cdots N\}$ of $\alpha(n)\times \alpha(n)$ random  matrices with the following property.  When $n\rightarrow \infty$, the sequence converges in distribution to the multidimensional free compound Poisson distribution determined by $\lambda$ and the distribution of $\{a_i\}_{i=1}^N.$
\end{Theorem}
\begin{proof}
Let $\{B(n, 1),..., B(n,N)\}$ be an almost sure random matrix model of $(a_1, ..., a_N)$. We can choose $B(n, i)$ to be an $n\times n$ random matrix in a probability space $(\Omega,P)$, for $i=1, ..., N$, and $n=1,2,...$.
It is obvious that there is a number  $p(n)\in \mathbb{N}$, for each $n\in \mathbb{N}$,  such that $p(n)\le n$ and $\lim_{n\rightarrow \infty}\frac{p(n)}{n}=\lambda$, for $n=1, 2, \cdots$, when $\lambda\le 1$. Actually, we can choose $p(n)$ as follows.
\begin{enumerate}
\item When $\lambda <1$, let $p(n)=\lambda n-\delta(n)$, where $0\le \delta(n)<1$. Then $\lim_{n\rightarrow \infty}\frac{p(n)}{n}=\lambda.$
\item When $\lambda=1$, let $p(n)=n$.
\end{enumerate}
Choose an $n\times n$ standard self-adjoint Gaussian matrix $X(n)=(X_{ij}(n))_{n^N\times n^N}$, which is independent from $\{\widetilde{B}(n,i):, i=1, 2,..., N\}$, where   $\widetilde{B}(n,i)=B(p(n), i)\oplus 0_{n-p(n)}$. For $m\in \mathbb{N}$,  and $\chi:\{1, 2, \cdots, m\}\rightarrow \{1, 2, \cdots, N\}$,  we have
\begin{align*}
&tr_{n}(\widetilde{B}(n,\chi(1), \omega)\cdots \widetilde{B}(n,\chi(m), \omega))\\
=&tr_{n}(B(p(n), \chi(1), \omega)\cdots B(p(n), \chi(m), \omega)\oplus 0_{n-p(n)})\\
=&\frac{p(n)}{n}tr_{p(n)}(B(p(n), \chi(1), \omega)\cdots B(p(n), \chi(m), \omega))\\
\rightarrow &\lambda \varphi(a_{\chi(1)}a_{\chi(2)}\cdots a_{\chi(m)}),
\end{align*}
 as $n\rightarrow \infty$, for almost all $\omega\in \Omega$.

 Now we show that $\{X(n)\widetilde{B}(n,i)X(n):i=1, 2, \cdots, N\} $ converges in distribution to the multidimensional compound free Poisson distribution with parameters $\lambda$ and the distribution $\mu_a$ of $a:=\{a_1, a_2, \cdots, a_N\}$.
By Theorem 5 of Section 4.2 in  \cite{MS}, $X(n)$ and $\{\widetilde{B}(n, i): i=1, 2, \cdots, N\}$ are almost surely asymptotically free, and the limit distribution of $(X(n))$ is the standard semicircular element $s$,  Hence, when $n\rightarrow\infty$, $\{X(n)\widetilde{B}(n,i)X(n):i=1, 2, \cdots, N\}$ converges in distribution to $\{sb_is:i=1, 2, \cdots, N\}$, where $s,b_1, \cdots, b_N$ are in a $C^*$-probability space, say, $(\A, \varphi)$, such that
\begin{enumerate}
\item $s$ is a standard semicircular element,
\item $\{b_i=b_i^*: i=1,2, \cdots, N\}$ has the almost sure limit distribution of $\{\widetilde{B}(n, 1), \cdots, \widetilde{B}(n, N)\}$,
\item $\{s\}$ and $\{b_i:i=1, 2, \cdots, N\}$ are free.
\end{enumerate}
 By Example 12.19 in \cite{NS}, for $1<m\in \mathbb{N}$, $\chi:\{1, 2, \cdots, m\}\rightarrow \{1, 2, \cdots, N\}$,
 $$\kappa_m(sb_{\chi(1)}s,\cdots, sb_{\chi(m)}s)=\varphi(b_{\chi(1)}\cdots b_{\chi(m)})=\lambda\varphi(a_{\chi(1)}\cdots a_{\chi(m)}).$$ This shows that $\{sb_1s, sb_2s, \cdots, sb_Ns\}$ has the desired multidimensional compound free Poisson distribution.

For $\lambda>1$, let $\lambda=k+\delta$, where $k\in \mathbb{N}$ and $0\le \delta<1$. Let $$\{X(n)\widetilde{B}(n,i)X(n):i=1, 2, \cdots, N\},  \{X(n)\widetilde{C}(n,i)X(n):i=1, 2, \cdots, N\}$$ be  random matrix models of the multidimensional free Poisson distributions $P(1, \mu_a)$ and $P(\delta, \mu_{a})$, respectively, where $P(\lambda, \mu_a)$ is the multidimensional free Poisson distribution determined by parameter $\lambda>0$ and the distribution $\mu_a$ of the tuple $a=(a_1, a_2, \cdots, a_N)$. Furthermore, we place two families $\{\widetilde{B}(n,i):i=1, 2, \cdots, N\}\subset M_{n}(\mathbb{C})$ and $\{\widetilde{C}(n,i):i=1, 2, \cdots, N\}\subset M_{n}(\mathbb{C})$ into the tensor product algebra $M_{n^{2}}(\mathbb{C})=M_{n}(\mathbb{C})\otimes M_{n}(\mathbb{C})$ so that the two families are tensorial independent,  and, therefore, $\{\widetilde{B}(n,i), \widetilde{C}(n,i): i=1, 2, \cdots, N\} $ has an almost sure limit distribution, as $n\rightarrow \infty$. Let $\widetilde{X}(n)$ be an $n^{2}\times n^{2}$ standard self-adjoint Gaussian random matrix, which is independent from $\{\widetilde{B}(n,i), \widetilde{C}(n,i): i=1, 2, \cdots, N\}$, for $n\in \mathbb{N}$. Then $\{\widetilde{X}(n)\widetilde{B}(n,i)\widetilde{X}(n):i=1, 2, \cdots, N\}$ and $\{\widetilde{X}(n)\widetilde{C}(n,i)\widetilde{X}(n):i=1, 2, \cdots, N\}$ are random matrix models for the multidimensional free Poisson distributions $P(1, \mu_a)$ and $P(\delta, \mu_{a})$, respectively.  By Theorem 5 in Section 4.2 in \cite{MS}, $\{\widetilde{X}(n)\}$ and  $\{\widetilde{B}(n,i), \widetilde{C}(n,i): i=1, 2, \cdots, N\} $ are almost surely asymptotically free. It implies that $$\{\widetilde{X}(n)\widetilde{B}(n,i)\widetilde{X}(n), \widetilde{X}(n)\widetilde{C}(n,i)\widetilde{X}(n): i=1, 2, \cdots, N\} \subset (M_{n^{2}}(L^{-\infty}(\Omega, P)), E\circ tr)$$ has an almost surely limit distribution as $n\rightarrow \infty$.

  Let $l=\max\{2N, k+1\}$,  $D_{n,i}=\widetilde{X}(n)\widetilde{B}(n,i)\widetilde{X}(n)$, for $i=1, 2, \cdots, N$, $D_{n,N+i}=\widetilde{X}(n)\widetilde{C}(n,i)\widetilde{X}(n)$, for $i=1, 2, \cdots, N$, and $D_{n, 2N+i}=I$, the identity matrix with an appropriate size, if $l>2N$. By the above discussions, the family $$\mathbf{D}(n)=\{D_{n, i}: i=1, 2, \cdots, l\}\subset (M_{n^{2}}(L^{-\infty}(\Omega, P)), E\circ tr)$$  converges  in distribution to $\{sb_1s, \cdots, sb_Ns, sc_1s, \cdots, sc_Ns,  1\}$ almost surely, where $\{sc_1s, \cdots, sc_Ns\}$ is the tuple of limit random variables of $\{\widetilde{X}(n)\widetilde{C}(n, i)\widetilde{X}(n):i=1, 2, \cdots, N\}$. Moreover, $$\lim_{n\rightarrow \infty}tr(D_{n,i}^{2m})\le \max\{2^{2m}\varphi((sb_is)^{2m}), 2^{2m}\varphi((sc_is)^{2m}), 1:i=1, 2, \cdots, N\}$$
  $$\le \max\{2\|sb_is\|, 2\|sc_is\|: i=1, 2, \cdots, N\}, \{1\}\}^{2m}, a. s., $$ for $i=1, 2, \cdots, l$. Therefore, we can choose   $D>0$ such that, for $m=1, 2, \cdots$,  $$\sup\{tr(D_{n, i}^{2m}): n=1, 2, \cdots, i=1, 2, \cdots, l\}\le D^{2m},a. s. $$
Let $\mathbf{U}_n=\{U_{n,i}:i=1, 2, \cdots, l\}$ be independent unitary random matrices with the Haar law, independent from $\mathbf{D}(n)$. By Theorem 4.5.10 in \cite{AGZ} and the discussion at the end of Section 4.3 in \cite{MS} (Page 110 of \cite{MS}), $$\{U_{n, 1}, U_{n, 1}, \cdots, U_{n, l}, U_{n, l}^*, \widetilde{X}(n)\widetilde{B}(n, i)\widetilde{X}(n),\widetilde{ X}(n)\widetilde{C}(n, i)\widetilde{X}(n), i=1, 2, \cdots, N, I_{n^{2N}}  \}$$ converges in distribution to $\{u_1, u_1^*, \cdots, u_{l}, u_{l}^*, sb_1s, \cdots, sb_Ns, sc_1s, \cdots, sc_Ns, 1\}$, as $n\rightarrow \infty$, $$\{u_1, u_1^*\}, \cdots, \{u_{l}, u_{l}^*\}, \{sb_is, sc_is: i=1, 2, \cdots, N, 1\} $$ are free, and $u_1, u_2, \cdots, u_{l}$ are Haar unitaries. Let $$Y_{n,j}=\sum_{i=1}^kU_{n, i}\widetilde{X}(n)\widetilde{B}(n, j)\widetilde{X}(n)U_{n,i}^*+U_{n, k+1}\widetilde{X}(n)\widetilde{C}(n, j)\widetilde{X}(n)U_{n,k+1}^*,$$ for $j=1, 2, \cdots, N$. Then $\{Y_{n, 1}, \cdots, Y_{n, N} \}$ converges in distribution to $\{y_1, \cdots, y_N\}$, where $$y_j=\sum_{i=1}^ku_isb_jsu_i^*+u_{k+1}sc_jsu_{k+1}^*,$$ for $j=1, 2, \cdots, N$.  We shall show that $\{y_1, \cdots, y_N\}$ has the multidimensional compound free Poisson distribution determined by $\lambda$ and the distribution $\mu_a$ of the tuple $a=\{a_1, a_2, \cdots, a_N\}$.

By the discussion before Theorem 9 in Section 4.3 of \cite{MS}, we have the following result. If $\{u_1, u_1^*\}, \{u_2, u_2^*\}$, and $\{d_1, \cdots, d_m\}$ are free, and $u_1$ and $u_2$ are Haar unitaries, then $$\{u_1d_1u_1^*, \cdots, u_1d_mu_1^*\},\ \{u_2d_1u_2^*, \cdots, u_2d_mu_2^*\}$$ are free. In fact, for polynomials $P_1, Q_1, \cdots, P_r, Q_r$ such that $$\varphi(P_i(u_1d_1u_1^*, \cdots, u_1d_mu_1^*))=\varphi(Q_i(u_2d_1u_2^*, \cdots, u_2d_mu_2^*))=0,$$ for $i=1, 2, \cdots, r$, we have $$\varphi(P_i(u_1d_1u_1^*, \cdots, u_1d_mu_1^*)))=\varphi(u_1(P_i(d_1, \cdots, d_m))u_1^*)=\varphi(P_i(d_1, \cdots, d_m))=0.$$
 Similarly, $\varphi(Q_i(d_1, \cdots, d_m))=0$, for $i=1, 2, \cdots, r$. Therefore,
\begin{align*}
&\varphi(P_1(u_1d_1u_1^*, \cdots, u_1d_mu_1^*)Q_1(u_2d_1u_2^*, \cdots, u_2d_mu_2^*)\\
\cdots &P_r(u_1d_1u_1^*, \cdots, u_1d_mu_1^*)Q_r(u_2d_1u_2^*, \cdots, u_2d_mu_2^*))\\
=&\varphi(u_1P_1(d_1, \cdots, d_m)u_1^*u_2Q_1(d_1, \cdots, d_m)u_2^*\\
\cdots& u_1P_r(d_1, \cdots, d_m)u_1^*u_2Q_r(d_1, \cdots, d_m)u^*_2)=0.
\end{align*}
It implies that $$\{u_1sb_1su_1^*, \cdots, u_1sb_Nsu_1^*\}, \cdots \{u_ksb_1su_k^*, \cdots, u_ksb_Nsu_k^*\}\}, \{u_{k+1}sc_1su_{k+1}^*, \cdots, u_{k+1}sc_Nsu_{k+1}^*\}$$ are free. Then for $m\in \mathbb{N}, \chi:\{1, 2, \cdots, m\}\rightarrow \{1, 2, \cdots, N\}$, we have
\begin{align*}
\kappa_m(y_{\chi(1)},\cdots,  y_{\chi(m)})=&\sum_{i=1}^k\kappa_m(u_isb_{\chi(1)}su_i^*, \cdots, u_isb_{\chi(m)}su_i^*)\\
+&\kappa_m(u_{k+1}sc_{\chi(1)}su_{k+1}^*, \cdots, u_{k+1}sc_{\chi(m)}su_{k+1}^*)\\
=&(k+\delta)\varphi(a_{\chi(1)}\cdots a_{\chi(m)})=\lambda\varphi(a_{\chi(1)}\cdots a_{\chi(m)}).
\end{align*}
\end{proof}

Recall from \cite{DV}, \cite{VDN}, and \cite{PS1} that two elements $a$ and b in non-commutative probability space $(\A, \varphi)$ are classically independent (or, tensor-independent) if $\varphi(a^{p_1}b^{q_1}\cdots a^{p_n}b^{q_n})=\varphi(a^{p})\varphi(b^{q})$, where $p=\sum_{i=1}^np_i$, $q=\sum_{i=1}^nq_i$, $0\le p_1, q_1, ..., p_n, q_n$ are integers. Generally, random variables    $a_1, a_2, \cdots, a_N$ are classically  independent, if for $n\in \mathbb{N}$, $\chi:\{1, 2, \cdots, n\}\rightarrow \{1, 2, \cdots, N\}$, and $\ker\chi$,  the partition of $\{1, 2, \cdots, n\}$ defined by $p\sim_{\ker\chi}q$ if and only if $\chi(p)=\chi(q)$, for $1\le p, q\le n$, we have $$\varphi(a_{\chi(1)}\cdots a_{\chi(n)})=\varphi_{\ker\chi}(a_{\chi(1)}\cdots a_{\chi(n)})=\prod_{V\in \ker\chi}\varphi(a_{\chi(V)}^{|V|}),$$ where $\chi(V)$ is the common value of $\chi$ when restricted to $V$.

\begin{Corollary} Let $a_1,..., a_N$ be self-adjoint random variables in a $C^*$-probability space $(\A, \varphi)$, and $\lambda$ a positive real number. If $a_1,..., a_N $ are  classically independent in $(\A, \varphi)$, then the compound free Poisson distribution $\pi_{\lambda, \mu_a}$ determined by $\lambda$ and the distribution $\mu_a$ of the tuple $a=(a_1,..., a_N)$ has a random matrix model.
\end{Corollary}
\begin{proof}
By the proof of Proposition 4.4.9 of \cite{HP}, for each $n\in \mathbb{N}$, there are real numbers $\xi_1(n,i)\le \cdots \le \xi_n(n,i)$ in the spectrum $\sigma(a_i)$ of $a_i$ such that the sequence $$\{B(n,i)=\mathbf{Diag}(\xi_1(n, i), \xi_2(n,i), \cdots, \xi_n(n,i)): n=1, 2, \cdots\}\subset (M_{n}(\mathbb{C}), tr)$$ converges in distribution to $a_i$ (see the bottom of Page 169 of the book, also Remark 22.27 in \cite{NS}), for $i=1, 2, \cdots, N$, where $tr$ is the normalized trace on the matrix algebra $M_n(\mathbb{C})$. Place $B(n,1), B(n,2)$, $\cdots, B(n,N)$ into the tensor product matrix algebra $M_{n^N}(\mathbb{C})=M_n(\mathbb{C})\otimes M_n(\mathbb{C})\cdots \otimes M_n(\mathbb{C})$ so that $$B(n,1), B(n,2), \cdots, B(n,N)$$ are mutually tensorial independent in the tensor product algebra. It implies that $$\{B(n,1), B(n,2), \cdots, B(n,N)\}$$ converges in distribution to $\{a_1, \cdots, a_N\}$, as $n\rightarrow \infty$. We have proved that $\{a_1, ..., a_N\}$ has an almost sure random matrix model. Now the conclusion follows from Theorem 4.2.
\end{proof}
Combining Theorem 4.2 and the work in Section 4 of \cite{PS}, we can get the following bi-random matrix model for a kind of bi-free compound free Poisson distributions, which is an illustrative example of the spirit of Section 4 of \cite{PS}: one can generalize results in random matrix models to bi-random matrix models.
\begin{Theorem} Let $\lambda>0$ and $a:=\{a_{1,l}, \cdots, a_{N, l},a_{1,r}, \cdots, a_{M, r} \}$ be a finite sequence of self-adjoint elements in a $C^*$-probability space $(\A, \varphi)$.  If the tuple $a:=\{a_{l,1}, \cdots, a_{l,N},a_{r,1}, \cdots, a_{r,M} \}$ has an almost sure random matrix model, and $a_{i, l}a_{j,r}=a_{j,r}a_{i, l}$ for $i=1, ...., N$ and $j=1, ..., M$,  then there exist a subsequence $\alpha(n)$ of natural numbers, and, for each $n\in \mathbb{N}$,  an $\alpha(n)\times \alpha(n)$ random two-faced family $Z(n):=((Z_{n,  i, l})_{1\le i\le N}, (Z_{n,j, r})_{1\le j\le M})$ of matrices such that $Z(n)$ converges in distribution to the bi-free compound free Poisson distribution determined by  $\lambda$ and $\mu_a$, the distribution of $a$.
\end{Theorem}
\begin{proof}
By the proof of Theorem 4.2, there exist a subsequence $\tilde{\alpha}(n)$ of natural numbers and a sequence $W_n:=\{W_{n,1}, W_{n, 2}, \cdots, W_{n, N+M}\}_{n=1}^\infty$ of $\tilde{\alpha}(n)\times \tilde{\alpha}(n)$ random matrices such that $$\{W_{n, i}\}_{i=1}^{N+M}\subset (M_{\tilde{\alpha}(n)}(L^{\infty-}(\Omega, P)), E\circ tr)$$ converges in distribution to the multidimensional compound free Poisson distribution $P(\lambda, \mu_a)$, which means  that, for $c_i=a_{l,i}$, for $i=1,..., N$ and $c_{N+i}=a_{r,i}$  $i=1,..., M$, $m\in \mathbb{N}$,  and $\alpha:\{1, 2, \cdots, m\}\rightarrow \{1, 2, \cdots, N+M\}$, we have $$\lim_{n\rightarrow\infty}\kappa_m(W_{n,\alpha(1)}, \cdots, W_{n,\alpha(m)})=\lambda\varphi(c_{\alpha(1)}\cdots c_{\alpha(m)}).$$

Define $X_{n, i, l}=L(W_{n, i}), X_{n, i, r}=R(W_{n,i}),$ for $i=1, 2, \cdots, N+M$.
We need the following result from \cite{PS}.
\vskip 3mm
\textit{ Remark 3.2 in \cite{PS}.  If $[T_{i,j}], [S_{ij}]\in M_N(\L(\X))$ are such that $T_{ij}S_{k,m}=S_{km}T_{ij}$, for all $i,j,k,m$, then it is elementary to verify that $L([T_{ij}])R([S_{ij}])=R([S_{ij}])L([T_{ij}])$}.
\vskip 3mm

By Remark 3.5 in \cite{PS}, for a non-commutative probability space $(\A, \varphi)$, $\A\subset \L(\A)$ by left multiplication.  It implies form Remarks 3.2 and 3.5 in \cite{PS} that$$L([T_{ij}])R([S_{ij}])=R([S_{ij}])L([T_{ij}]),$$ if $[T_{ij}], [S_{ij}]\in M_N(\A)\subset M_N(\L(\A))$, where $\A$ is a commutative algebra. It follows that $$X_{n, i, l}X_{ n, j, r}=X_{ n, j, r}X_{ n, i, l},  i, j=1, 2, \cdots, N+M,\eqno (4.1)$$ since the entries of the matrices $W_{n,i}$, for $i=1, 2, ..., M+N$, are  elements in the commutative algebra $L^{\infty-}(\Omega, P)$.

 Let $m\in \mathbb{N}$, $\chi: \{1, ..., m\}\rightarrow \{l,r\}$, and $\alpha:\{1,..., m\}\rightarrow \{1, ..., N+M\}$. We define a permutation $s\in S_{m}$ by the following equations $$\chi^{-1}(\{l\})=\{s(1)<\cdots<s(m-k)\},  \chi^{-1}(\{r\})=\{s(m-k+1)<\cdots <s(m)\}.$$
We prove $$\lim_{n\rightarrow \infty}\kappa_{\chi}(X_{n, \alpha(1),\chi(1)}, \cdots, X_{n, \alpha(m), \chi(m)})=\lambda\varphi(c_{\alpha(s(1))}\cdots c_{\alpha(s(m))}).\eqno (4.2)$$
 Suppose  $|\chi^{-1}(\{r\})|=0$, that is,  $\chi(i)=l$, for all $i=1, 2, \cdots, m$. In this case, $s$ is the identity permutation of $S_m$, that is, $s(1)=1, ..., s(m)=m$.
By Lemma 3.7 in \cite{PS}, in this case, the distribution of $\{X_{n, \alpha(1),\chi(1)}, ..., X_{n,\alpha(m),\chi(m)}\}$ in $(\L(M_{\alpha(n)}(L^{\infty-}(\Omega, P))), \Phi_n:=tr\circ E_{\L(M_{\alpha(n)}(L^{\infty-}(\Omega, P)))})$ is equal to  that of $\{W_{n, \alpha(1)}, \cdots, W_{n, \alpha(m)}\}$ in $(M_{\alpha(n)}(L^{\infty-}(\Omega, P)), E\circ tr)$. It implies from the discussion in first paragraph of this proof that
 $$\lim_{n\rightarrow \infty}\kappa_{\chi}(X_{n, \alpha(1),\chi(1)}, \cdots, X_{n, \alpha(m),\chi(m)})=\lim_{n\rightarrow \infty}\kappa_m(W_{n, \alpha(1)}, \cdots, X_{n, \alpha(m)})=\lambda\varphi(c_{\alpha(1)}\cdots c_{\alpha(m)}).$$

Suppose (4.2) holds true when $|\chi^{-1}(\{r\})|=k$. Now we prove $(4.2)$ when $|\chi^{-1}(\{r\})|=k+1$. We adopt some ideas from Example 4.15 in \cite{PS}.
Let $\hat{\chi}=\chi\circ s$, where $s$ is the function defined before $(4.2)$.
Then $\hat{\chi}^{-1}(\{l\})=\{1, 2, \cdots, m-k-1\}$ and $\hat{\chi}^{-1}(\{r\})=\{m-k, \cdots, m\}$. Define a lattice isomorphism $\rho:BNC(\chi)\rightarrow BNC(\hat{\chi}), \rho:\pi\mapsto s^{-1}\circ \pi$. Then $\mu_{BNC}(\pi, 1_\chi)=\mu_{BNC}(s^{-1}\circ \pi, 1_{\hat{\chi}})$.
  It follows by $(4.1)$ that $$\Phi_n(X_{n, \alpha(1),\chi(1)}, ..., X_{n,\alpha(m),\chi(m)})=\Phi_{n}(X_{n,\alpha(s(1)), \hat{\chi}(1)}, ...,X_{n, \alpha(s(m)), \hat{\chi}(m)}).\eqno (4.3)$$ For a partition $V\in \pi\in BNC(\chi)$, we have $ s^{-1}\circ V\in s^{-1}\circ \pi\in BNC(\hat{\chi})$. Without loss of generality, we can assume that $V=\{p_1\prec_{\chi}\cdots\prec_{\chi}p_q\}$ is a $\chi$-ordered interval of the $\chi$-ordered set $$\{s(1)\prec_{\chi}\cdots \prec_{\chi}s(m-k-1)\prec_{\chi}s(m)\prec_{\chi}\cdots \prec_{\chi}s(m-k)\}.$$ Then $s^{-1}\circ V=\{s^{-1}(p_1)\prec_{\hat{\chi}}\cdots \prec_{\hat{\chi}}s^{-1}(p_q)\}$ is a $\hat{\chi}$-ordered interval of the $\hat{\chi}$-ordered set $$\{1\prec_{\hat{\chi}} 2\prec_{\hat{\chi}} \cdots \prec_{\hat{\chi}}m-k-1\prec_{\hat{\chi}} m \prec_{\hat{\chi}} \cdots \prec_{\hat{\chi}}m-k\}.$$  By $(4.3)$, we get
  $$\Phi_V(X_{n, \alpha(1),\chi(1)}, ..., X_{n,\alpha(m),\chi(m)})=\Phi_{s^{-1}\circ V}(X_{n,\alpha(s(1)), \hat{\chi}(1)}, ...,X_{n, \alpha(s(m)), \hat{\chi}(m)}).$$ It implies that
$$\Phi_\pi(X_{n, \alpha(1), \chi(1)}, ..., X_{n, \alpha(m), \chi(m)})=\Phi_{s^{-1}\circ \pi}(X_{n,  \alpha(s(1)),\hat{\chi}(1)}, ..., X_{n, \alpha(s(m)), \hat{\chi}(m)}),$$ for $\pi\in BNC(\chi)$. We thus have $$\kappa_{\chi}(X_{n, \alpha(1), \chi(1)}, ..., X_{n, \alpha(m), \chi(m)})
= \sum_{\pi\in BNC(\chi)}\Phi_{n,\pi}(X_{n, \alpha(1), \chi(1)},..., X_{n, \alpha(m), \chi(m)})\mu_{BNC}(\pi, 1_\chi)$$
$$=\sum_{\pi\in BNC(\hat{\chi})}\Phi_{n, \pi}(X_{n, \alpha(s(1)),\hat{\chi}(1)},..., X_{n, \alpha(s(m)), \hat{\chi}(m)})\mu_{BNC}(\pi, 1_{\hat{\chi}}). \eqno (4.4)$$

We then define $\tilde{\chi}:\{1, 2, \cdots, m\}\rightarrow \{l,r\}$ by $\tilde{\chi}(i)=\hat{\chi}(i)$, if $i<m$; $\tilde{\chi}(m)=l$. Replacing $\hat{\chi}$ by $\tilde{\chi}$ induces an isomorphism from $BNC(\hat{\chi})$ to $BNC(\widetilde{\chi})$. It follows from the fact $X_{n, p, r}I_{\alpha(n)}=X_{n, p, l}I_{\alpha(n)}$ for $p=1, 2,..., N+M $ and $n\in \mathbb{N}$, and $(4.4)$ that
$$\kappa_{\chi}(X_{n, \alpha(1), \chi(1)}, ..., X_{n, \alpha(m), \chi(m)})=\sum_{\pi\in BNC(\tilde{\chi})}\Phi_{n, \pi}(X_{n, \alpha(s(1)), \tilde{\chi}(1)},..., X_{n, \alpha(s(m)),\tilde{\chi}(m)})\mu_{BNC}(\pi, 1_{\tilde{\chi}}). \eqno(4.5)$$
By the inductive hypothesis and  $(4.5)$, we have

$$\lim_{n\rightarrow \infty}\kappa_{\chi}(X_{n, \alpha(1), \chi(1)}, ..., X_{n, \alpha(m), \chi(m)})
=\lim_{n\rightarrow \infty}\kappa_{\tilde{\chi}}(X_{n, \alpha(s(1)), \tilde{\chi}(1)}, ..., X_{n, \alpha(s(m)), \tilde{\chi}(m)})$$
$$=\lambda\varphi(c_{\alpha(s(1))}\cdots c_{\alpha(s(m))}),
$$
where the last equality holds true because $|\tilde{\chi}^{-1}(\{r\})|=|\hat{\chi}(\{r\})|-1=k$. We have proved $(4.2)$.

For  $m\in \mathbb{N}$,   $\chi:\{1, ..., m\}\rightarrow \{l,r\}$, and  function $\alpha:\{1, ..., m\}\rightarrow \{1, ..., N+M\}$ such that $1\le \alpha(i)\le N$ if $\chi(i)=l$, and $N<\alpha(i)\le N+M$ if $\chi(i)=r$, for $i=1, ..., m$, Then $c_{\alpha(i)}=a_{l,\alpha(i)}$ if $\alpha(i)\le N$; $c_{\alpha(i)}=a_{r, \alpha(i)-N}$, if $\alpha(i)>N$, for $i=1, ..., m$. It follows that $$\varphi(c_{\alpha(1)}\cdots c_{\alpha(m)})=\varphi(c_{\alpha(s(1))}\cdots c_{\alpha(s(m))}),\eqno (4.6)$$ because the left random variables commute with the right random variables in the family $$a=((a_{l, i})_{1\le i\le N}, (a_{r, j})_{1\le j\le M}).$$ By $(4.2)$ and $(4.6)$, we have
$$\lim_{n\rightarrow \infty}\kappa_{\chi}(X_{n, \alpha(1),\chi(1)}, \cdots, X_{n, \alpha(m), \chi(m)})
=\lambda \varphi(c_{\alpha(1)}\cdots c_{\alpha(m)}),$$
which shows that  $$\{(X_{ n, i, l})_{1\le i\le N}, (X_{ n, i, r})_{N+1\le i\le N+M}\} \subset (\L(M_{\alpha(n)}(L^{\infty-}(\Omega, P))), \Phi_n:=tr\circ E_{\L(M_{\alpha(n)}(L^{\infty-}(\Omega, P)))})$$ converges in distribution to the compound bi-free  Poisson distribution determined by $\lambda$ and the distribution of $\mu_a$.

Let $Z_{n,i,l}=X_{n, i, l}$ for $i=1,..., n$ and $Z_{n, i, r}=X_{n, N+i, r}$ for $i=1,..., M$. Then the random two faced family $Z_n:=((Z_{n, i,l})_{1\le i\le N}, (Z_{n, i,r})_{1\le i\le M})$ of matrices converges in distribution to the compound bi-free Poisson distribution determined by $\lambda$ and the distribution $\mu_a$.

\end{proof}
\section{Bi-Matrix Models With Fock Space Entries }

This section is devoted to constructing a bi-matrix model for a compound bi-free  Poisson distribution determined by a positive number and a commutative pair of random variables, where the bi-matrix model consists of  matrices with entries of creation and annihilation operators on the full Fock space of a Hilbert space. This bi-matrix model is an analogue of P. Skoufranis' similar bi-matrix models for bi-free central limit distributions in \cite{PS}.

\begin{Theorem} Let  $\lambda>0$, $\{a_1, a_2\}$ be a commutative pair of  random variables in a non-commutative space $(\A, \varphi)$.  Then for $n\in N$, there is  a sequence $\{(Z_{n, N, l}, Z_{n,N, r}): N=1, 2, \cdots\}$ of $n\times n$ left and right matrices $Z_{n, N, l}$ and $Z_{n, N, r}$, respectively, with entries of creation and annihilation operators on a full fuck space such that $Z_{n,N}=(Z_{n,N, l}, Z_{n, N, r})$ converges in distribution to the compound bi-free Poisson distribution determined by $\lambda$ and $\mu_{a_1, a_2}$, the distribution of the pair $(a_1, a_2)$, as $N\rightarrow \infty$. That is,   $$\lim_{N\rightarrow \infty}\kappa_\chi(Z_{n, N})=\lambda\varphi(a_{\chi(1)}\cdots a_{\chi(m)})=\lambda\varphi(a_1^{|\chi^{-1}(\{l\})|}a_2^{|\chi^{-1}(\{r\})|}), $$ for $\chi:\{1, 2, \cdots, m\}\rightarrow \{1,2\}$ and $m\ge 1$.
\end{Theorem}
\begin{proof}
Let $\X=\F(\H)$ be the full Fock space of an infinite dimensional complex Hilbert space $\H$. Let $\X^0=\bigoplus_{n\ge 1}\H^{\otimes n}$,then $\X=\mathbb{C}\Omega\oplus \X^0$. Let  $p:\X\rightarrow \mathbb{C},\  p (\lambda\Omega+x_0)=\lambda$. Then  $(\X, \X^0, \Omega, p)$ is a pointed vector space. Define a unital linear functional $\omega: \L(\X)\rightarrow \mathbb{C}, \omega(T)=\langle T\Omega, \Omega\rangle$, for $T\in \L(\X)$.
Let $\{e_1, e_2\}$ be an orthonormal set in $\H$. Let $l_i$ and $r_i$ be the left and right creation operators associated with $e_i$, respectively, for $i=1, 2$.

Let $$W_{N,,i,l}=l^*_i+\sum_{n=0}^N \sum_{\alpha:\{1, 2, \cdots, n\}\rightarrow \{1, 2\}}\lambda \varphi(a_{\alpha(1)}\cdots a_{\alpha(n)}a_i)l_{\alpha(n)}\cdots l_{\alpha(1)}, $$ $$W_{N, i, r}=r_i^*+\sum_{n=0}^N \sum_{\alpha:\{1, 2, \cdots, n\}\rightarrow \{1,2\}}\lambda \varphi(a_{\alpha(1)}\cdots, a_{\alpha(n)}a_i)r_{\alpha(n)}\cdots r_{\alpha(n)},$$ for $i=1, 2$, and $N\in \mathbb{N}$. For $m\in \mathbb{N}$, and $\chi:\{1, 2, \cdots, m\}\rightarrow \{l, r\}$, we define $$ W_{N, l}=W_{N, 1,l}, W_{N, r}=W_{N, 2, r}.$$ Then by Theorem 6.2  in \cite{MN}, we have
$$\kappa_\chi(W_{N,\chi(1)}, \cdots, W_{N, \chi(m)})=\left\{\begin{array}{ll}\lambda \varphi(a_1^{|\chi^{-1}(\{l\})|}a_2^{|\chi^{-1}(\{r\})|}),&\text{if } m\le N,\\
0, & \text{if } m>N.\end{array} \right.$$

 For $n\in \mathbb{N}$,  let  $\{e_{i,j}^k: i, j=1, 2, \cdots, n, k=1, 2\}$ be an orthonormal set of $\H$, and $$L_{k}=\frac{1}{\sqrt{N}}L([l(e_{i,j}^k)]_{n\times n}), L_k^*=\frac{1}{\sqrt{N}}L([l^*(e_{j,i}^k)]_{n\times n}), $$ $$R_k=\frac{1}{\sqrt{N}}R([r(e_{i,j}^k)]_{n\times n}), R_k^*=\frac{1}{\sqrt{N}}L([r^*(e_{j,i}^k)]_{n\times n}), k=1, 2, $$
 where $L(A)$ and $R(A)$ are left and right matrices associated with matrix $A$, respectively, and $l(e_{i,j}^k)$ and $r(e_{i,j}^k)$ are left and right creation operators on $\X$ associated with vector $e_{i,j}^k$. By Theorem 5.1 in \cite{PS}, the joint distribution of $\{L_k, L_k^*, R_k, R_k^*: k=1, 2\}$ with respect to $\Phi:=tr\circ E_{\L(\X_n)}$ is equal to the joint distribution  of $\{l_k, l_k^*, r_k, r_k^*: k=1, 2\}$ with respect to $\omega$. Let $$Z_{n, N, l}=L_1^*+\sum_{i=0}^N\sum_{\alpha:\{1, 2, \cdots, i\}\rightarrow \{1, 2\}}\lambda \varphi(a_{\alpha(1)}\cdots a_{\alpha(i)}a_1)L_{\alpha(i)}\cdots L_{\alpha(1)},$$
  $$Z_{n, N, r}=R_2^*+\sum_{i=0}^N\sum_{\alpha:\{1, 2, \cdots, i\}\rightarrow \{1,2\}}\lambda\varphi(a_{\alpha(1)}\cdots a_{\alpha(i)}a_2)R_{\alpha(i)}\cdots R_{\alpha(1)}.$$ We then have \begin{align*}\kappa_\chi(Z_{n, N, \chi(1)}, \cdots, Z_{n, N, \chi(m)})=&\kappa_\chi(W_{N, \chi(1)}, \cdots, W_{N, \chi(m)})\\
  =&\left\{\begin{array}{ll}\lambda\varphi(a_1^{|\chi^{-1}(\{l\})|}a_2^{|\chi^{-1}(\{r\})|}) ,&\text{if } m\le N,\\
0, & \text{if } m>N.\end{array} \right.
\end{align*} It implies that $$\lim_{N\rightarrow \infty}\kappa_\chi(Z_{n, N, \chi(1)}, \cdots, Z_{n, N, \chi(m)})=\lambda \varphi(a_1^{|\chi^{-1}(\{l\})|}a_2^{|\chi^{-1}(\{r\})|}).$$
\end{proof}

\end{document}